\documentclass[graybox]{svmult}

\usepackage{mathptmx}       % selects Times Roman as basic font
\usepackage{helvet}         % selects Helvetica as sans-serif font
\usepackage{courier}        % selects Courier as typewriter font
\usepackage{type1cm}        % activate if the above 3 fonts are
\usepackage{makeidx}         % allows index generation
\usepackage{graphicx}        % standard LaTeX graphics tool
\usepackage{multicol}        % used for the two-column index
\usepackage[bottom]{footmisc}% places footnotes at page bottom
\usepackage{amsfonts}
\usepackage{amssymb}
\usepackage{algorithm}
\usepackage[noend]{algpseudocode}
\makeindex             % used for the subject index
\begin{document}

\title*{A parallel algorithm for the constrained shortest path problem on lattice graphs}
\author{Ivan Matic\\
Published in: Adamatzky, A (Ed.) Shortest path solvers. From software to wetware. Springer, 2018.}
\authorrunning{Ivan Matic}

\institute{Ivan Matic \at Department of Mathematics, Baruch College, CUNY, One Bernard Baruch Way, New York, NY 10010, USA \email{Ivan.Matic@baruch.cuny.edu}
}
%\author{Published in: Adamatzky, A (Ed.) Shortest path solvers. From software to wetware. Springer, 2018.}
\maketitle

\abstract{The edges of a graph are assigned weights and passage times which are assumed to be positive integers. We present a parallel algorithm for finding the shortest path whose total weight is smaller than a pre-determined value.  In each step the processing elements are not analyzing the entire graph. Instead they are focusing on a subset of vertices called {\em active vertices}. The set of active vertices at time $t$ is related to the boundary of the ball $B_t$ of radius $t$ in the first passage percolation metric.  Although it is believed that the number of active vertices is an order of magnitude smaller than the size of the graph, we prove that this need not be the case with an example of a graph for which the active vertices form a large fractal. We analyze an OpenCL implementation of the algorithm on GPU for cubes in $\mathbb Z^d$.
}

\section{Definition of the problem}
\label{sec:1}
The graph $G(V,E)$ is undirected and the function  $f:E\to\mathbb Z_+^2$ is defined on the set of its edges. The first component $f_1(e)$ of the ordered pair  $f(e)=\left(f_1(e),f_2(e)\right)$ for a given edge $e\in E$ represents the time for traveling over the edge $e$. The second component $f_2(e)$ represents the weight of $e$.

 A path in the graph $G$ is a sequence of vertices $(v_1, v_2, \dots, v_k)$ such that for each $i\in\{1,2,\dots, k-1\}$ there is an edge between $v_i$ and $v_{i+1}$, i.e. $(v_i,v_{i+1})\in E$. 
For each path $\pi = (v_1, \dots, v_k)$ we define $F_1(\pi)$ as the total time it takes to travel over  $\pi$ and $F_2(\pi)$ as the sum of the weights of all edges in $\pi$. Formally,
\begin{eqnarray*}
F_1(\pi)=\sum_{i=1}^{k-1} f_1\left(v_{i},v_{i+1}\right)\quad\mbox{and}\quad
F_2(\pi)=\sum_{i=1}^{k-1} f_2\left(v_i,v_{i+1}\right).
\end{eqnarray*}

Let $A, B\subseteq V$ be two fixed disjoint subsets of $V$ and let $M\in\mathbb R_+$ be a fixed positive real number.
 Among all paths that connect sets $A$ and $B$ let us denote by $\hat\pi$ the one (or one of) for which $F_1(\pi)$ is minimal under the constraint $F_2(\pi)< M$. We will describe an algorithm whose output will be $F_1(\hat \pi)$ for a given graph $G$. 

The algorithm belongs to a class of label correcting algorithms \cite{irnich_desaulniers, mehlhorn_ziegelmann}. The construction of labels will aim to minimize the memory consumption on SIMD devices such as graphic cards. Consequently, the output will not be sufficient to determine the exact minimizing path.
The reconstruction of the minimizing path is possible with subsequent applications of the method, because the output can include the vertex $X\in B$ that is the endpoint of $\hat \pi$, the last edge $x$ on the path $\hat \pi$, and the value $F_2(\hat \pi)$. 
Once $X$ and $x$ are found, the entire process can be repeated for the graph $G'(V',E')$ with $$V'=V\setminus B,\quad A'=A, \quad B'=\{X\}, \quad\mbox{ and }\quad M'=F_2\left(\hat \pi\right)-f_2(x).$$ The result will be second to last vertex on the minimizing path $\hat \pi$. All other vertices on $\hat\pi$ can be found in the same way. 

Although the algorithm works for general graphs and integer-valued functions $f$, its implementation on SIMD hardware requires the vertices to have bounded degree. This requirement is satisfied by subgraphs of $\mathbb Z^d$. 

 Finding the length of the shortest path in graph is equivalent to finding the shortest passage time in first passage percolation. Each of the vertices in $A$ can be thought of as a source of water. The value $f_1(e)$ of each edge $e$ is the time it takes the water to travel over $e$. Each drop of water has its {\em quality} and each drop that travels through edge $e$ looses $f_2(e)$ of its quality. Each vertex $P$ of the graph has a label $\mbox{\em Label}(P)$ that corresponds to the quality of water that is at the vertex $P$. Initially all vertices in $A$ have label $M$ while all other vertices have label $0$.  The drops that get their quality reduced to $0$ cannot travel any further. The time at which a vertex from $B$ receives its first drop of water is exactly the minimal $F_1\left(\pi\right)$ under the constraint $F_2(\pi)<M$. 

Some vertices and edges in the graph are considered {\em active}. Initially, the vertices in $A$ are {\em active}. All edges adjacent to them are also called {\em active}.
Each cycle in algorithm corresponds to one unit of time. During one cycle the water flows through active edges and decrease their time components by $1$. Once an edge gets its time component reduced to $0$, the edge becomes {\em used} and we look at the source $S$ and the destination $D$ of this water flow through the edge $e$. The destination $D$ becomes {\em triggered}, and its label will be {\em corrected}. The label correction is straight-forward if the edge $D$ was inactive. We simply check whether $\mbox{\em Label}(S)-f_2(e)>\mbox{\em Label}(D)$, and if this is true then the vertex $D$ gets its label updated to $\mbox{\em Label}(S)-f_2(e)$ and its status changed to {\em active}. If the vertex $D$ was active, the situation is more complicated, since the water has already started flowing from the vertex $D$. The existing water flows correspond to water of quality worse than the new water that has just arrived to $D$. We resolve this issue by introducing phantom edges to the graph that are parallel to the existing edges. The phantom edges will carry this new high quality water, while old edges will continue carrying their old water flows. A vertex stops being active if all of its edges become used, but it may get activated again in the future.

 \section{Related problems in the literature}
The assignment of phantom edges to the vertices of the graph and their removal is considered a label correcting approach in solving the problem. Our particular choice of label correction is designed for large graphs in which the vertices have bounded degree. Several existing serial computation algorithms can find the shortest path by maintaining labels for all vertices. The labels are used to store the information on the shortest path from the source to the vertex and additional preprocessing of vertices is used to achieve faster implementations \cite{boland_dethridge_dumitrescu, desrochers_desrosiers_solomon}. The ideas of first passage percolation and label correction have naturally appeared in the design of {\em pulse algorithms} for constrained shortest paths \cite{lozano_medaglia}. All of the mentioned algorithms can also be parallelized but this task would require a different approach in designing a memory management  that would handle the label sets in programming environments where dynamical data structures need to be avoided.

The method of aggressive edge elimination \cite{muhandiramge_boland} can be parallelized to solve the Lagrange dual problems. 
In the case of road and railroad networks a substantial speedup can be achieved by using a preprocessing of the network data and applying a generalized versions of Dijkstra's algorithm \cite{kohler_mohring_schilling}.

The parallel algorithm that is most similar in nature to the one discussed in this paper is developed for wireless networks \cite{li_wan_wang_frieder}. There are two features of wireless networks that are not available to our model. The first feature is that the communication time between the vertices can be assumed to be constant. The other feature is that wireless networks have a processing element available to each vertex. Namely, routers are usually equipped with processors. Our algorithm is build for the situations where the number of processing cores is large but not at the same scale as the number of vertices. On the other hand our algorithm may not be effective for the wireless networks since the underlying graph structure does not imply that the vertices are of bounded degree. 
The increase of efficiency of wireless networks can be achieved by solving other related optimization problems. One such solution is based on constrained node placement \cite{satyajayant_majd_huang}.  

The execution time of the algorithm is influenced by the sizes of the sets of active vertices, active edges, and phantom edges. The sizes of these sets are order of magnitude smaller than the size of the graph. Although this cannot be proved at the moment, we will provide a justification on how existing conjectures and theorems from the percolation theory provide some estimates on the sizes of these sets.  The set of active vertices is related to the limit shape in the model of first passage percolation introduced by Hammersley and Welsh \cite{hammersley_welsh}. The first passage percolation corresponds to the case $M=\infty$, i.e. the case when there are no constraints. If we assume that $A=\{0\}$, for each time $t$ we can define the ball of radius $t$ in the first passage percolation metric as: 
\begin{eqnarray*}B_t=\left\{x: \tau(0,x)\leq t\right\},\end{eqnarray*} where $\tau(0,x)$ is the  {\em first passage time}, i.e. the first time at which the vertex $x$ is reached.  

The active vertices at time $t$ are located near the boundary of the ball $B_t$. It is known that for large $t$ the set $\frac1tB_t$ will be approximately convex. More precisely, it is known \cite{cox_durrett} that there is a convex set $B$ such that 
\begin{eqnarray*}\mathbb P\left((1-\varepsilon)B\subseteq \frac1tB_t\subseteq (1+\varepsilon)B \mbox{ for large }t \right)=1.\end{eqnarray*}

However, the previous theorem does not guarantee that the boundary of $B_t$ has to be of zero volume. In fact the boundary can be non-polygonal as was previously shown \cite{damron_hochman}. 

The set of active vertices does not coincide with the boundary of $B_t$, but it is expected that if $\partial B_t$ is of small volume then the number of active vertices is small in most typical configurations of random graphs. We provide an example for which the set of active vertices is a large fractal, but simulations suggest that this does not happen in average scenario. 

The fluctuations of the shape of $B_t$ are expected to be of order $t^{2/3}$ in the case of $\mathbb Z^2$ and the first passage time $\tau(0,n)$ is proven to have fluctuations of order at least $\log n$  \cite{newman_piza}. The fluctuations are of order at most $n/\log n$ \cite{benjamini_kalai_schramm, benaim_rossignol} and are conjectured to be of order $t^{2/3}$. They can be larger and of order $n$ for modifications of $\mathbb Z^2$ known as thin cylinders \cite{chatterjee_dey}.

The scaling of $t^{2/3}$ for the variance is conjectured for many additional interface growth models and is related to the Kardar-Parisi-Zhang equation \cite{amir_corwin_quastel, krug_spohn, sasamoto_spohn}. 

The constrained first passage percolation problem is a discrete analog to Hamilton-Jacobi equation. The large time behaviors of its solutions are extensively studied and homogenization results are obtained for a class of Hamiltonians \cite{scott_hung_yifeng_2014, krv, kyz, souganidis_1999}. Fluctuations in dimension one are of order $t$ \cite{rezakhanlou_clt} while in higher dimensions they are of lower order although only the logarithmic improvement to the bound has been achieved so far \cite{matic_nolen}.

\section{Example}
\noindent
Before providing a more formal description of the algorithm we will illustrate the main ideas on one concrete example of a graph. 
Consider the graph shown in Figure \ref{fi:figure 1} that has $12$ vertices labeled as $1$, $2$, $\dots$, $12$. The set $A$ contains the vertices $1$, $2$, and $3$, and the set $B$ contains only the vertex $12$. The goal is to find the length of the shortest path from $A$ to $B$ whose total weight is smaller than $19$.

\begin{figure} 
\centering
\begin{center}
\includegraphics[scale=0.25]{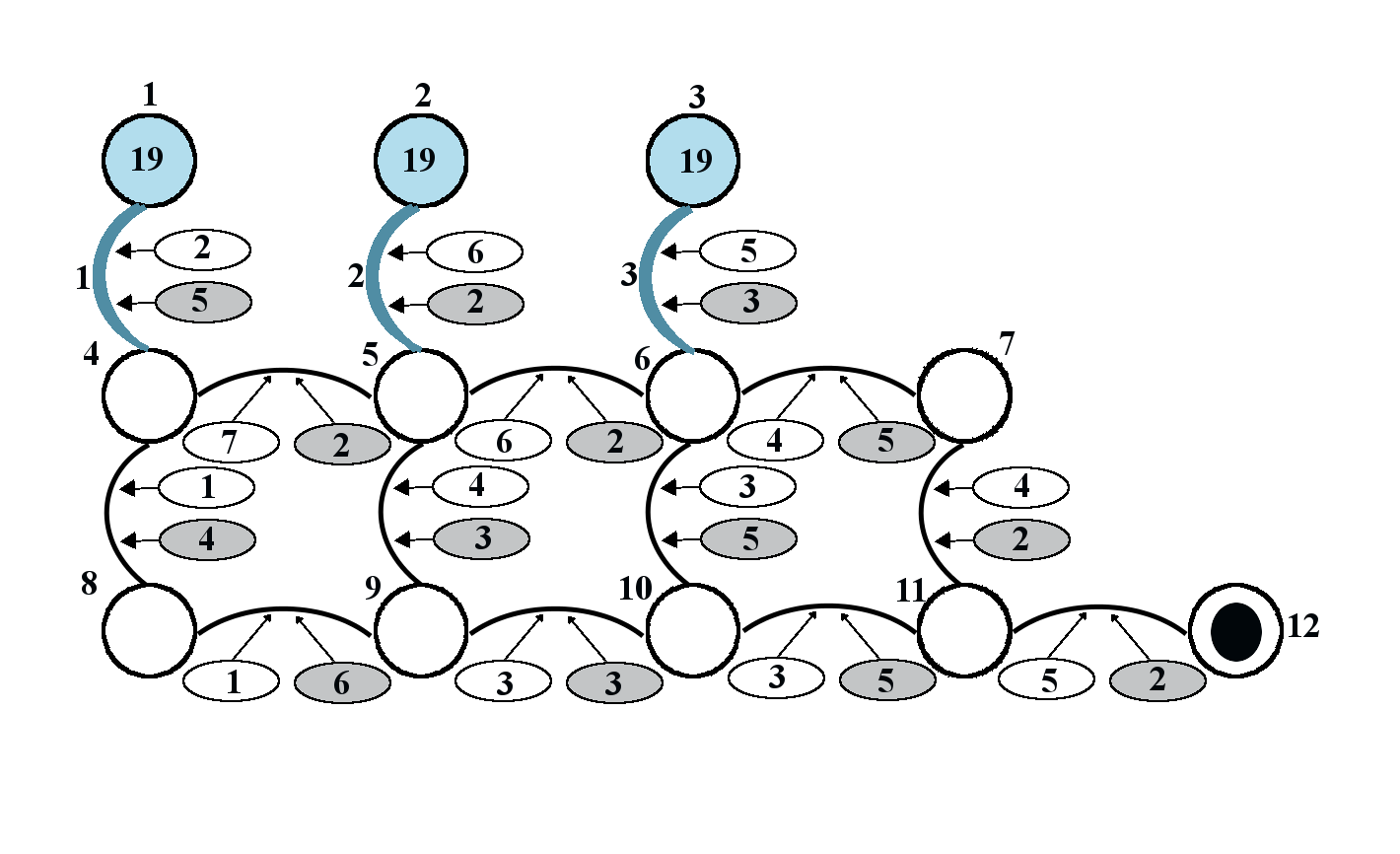}  
\end{center}
\caption{\label{fi:figure 1} The initial state of the graph.}
\end{figure}
The vertices are drawn with circles around them. The circles corresponding to the vertices in $A$ are painted in blue and have the labels $19$. The picture contains the time and weight values for each of the edges. The time parameter of each edge is written in the empty oval, while the weight parameter is in the shaded oval. Since the number of edges in this graph is relatively small it is not difficult to identify the the minimizing path $(3$, $6$, $10$, $11$, $12)$. The time required to travel over this path is $16$ and the total weight is $15$. 

Initially, the vertices in set $A$ are called {\em active}. Active vertices are of blue color and edges adjacent to them are painted in blue. These edges are considered {\em active}. Numbers written near their centers represent the sources of water. For example, the vertex $2$ is the source of the flow that goes through the edge $(2,5)$.

Notice that the smallest time component of all active edges is $2$. The first cycle of the algorithm begins by decreasing the time component of each active edge by $2$. The edge $(1,4)$ becomes {\em just used} because its time component is decreased to $0$. The water now flows from the vertex $1$ to the vertex $4$ and its quality decreases by $5$, since the weight of the edge $(1,4)$ is equal to $5$. The vertex $4$ becomes active and its label is set to $$\mbox{\em Label}(4)=\mbox{\em Label}(1)-f_2(1,4)=19-5=14.$$ The edge $(1,4)$ becomes used, and the vertex $1$ turns into inactive since there are no active edges originating from it. Hence, after two seconds the graph turns into the one shown in Figure \ref{fi:figure 2}. 
\begin{figure}
\centering
\begin{center}
\includegraphics[scale=0.25]{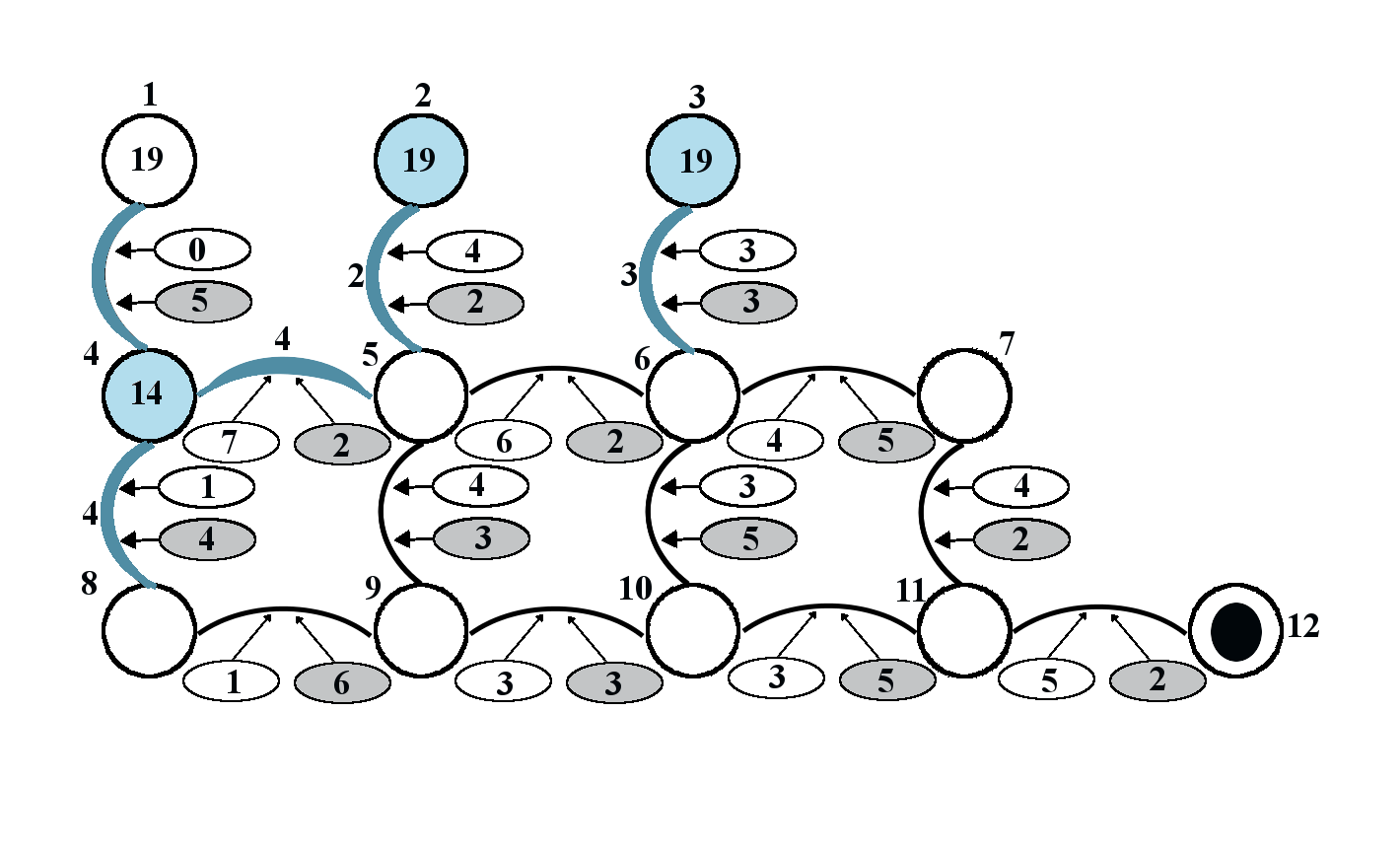}  
\end{center}
\caption{\label{fi:figure 2} The configuration after the second 2.}
\end{figure}

The same procedure is repeated until the end of the $5$th second and the obtained graph is the top one in Figure \ref{fi:figure 3}. In the $6$th second the edge $(2,5)$ gets its time parameter decreased to $0$ and the vertex $5$ gets activated. Its label becomes $$\mbox{\em Label}(5)=\mbox{\em Label}(2)-f_2\left(2,5\right)=19-2=17.$$ However, the edges $(4,5)$, $(5,9)$, and $(5,6)$ were already active and the water was flowing through them towards the vertex $5$. 

The old flow of water through the edge $(4,5)$ will complete in additional $5$ seconds. However, when it completes the quality of the water that will reach the vertex $5$ will be $$\mbox{\em Label}(4)-f_2\left(4,5\right)=14-2=12<\mbox{\em Label}(5)$$ because the label of the vertex $5$ is $17$. Thus there is no point in keeping track of this water flow. On the other hand, the water flow that starts from $5$ and goes towards $4$ will have quality $$\mbox{\em Label}(5)-f_2(4,5)=17-2=15$$ which is higher than the label of the vertex $4$. Thus the edge $(4,5)$ will change its source from $4$ to $5$ and the time parameter has to be restored to the old value $7$. At this point the vertex $4$ becomes inactive as there is no more flow originating from it.

The same reversal of the direction of the flow happens with the edge $(5,9)$. On the other hand, something different happens to the edge $(5,6)$: it stops being active. The reason is that the old flow of water from $6$ to $5$ will not be able to increase the label of the vertex $5$. Also, the new flow of water from $5$ to $6$ would not be able to change the label of vertex $6$. 
\begin{figure}
\centering
\begin{center}
\includegraphics[scale=0.20]{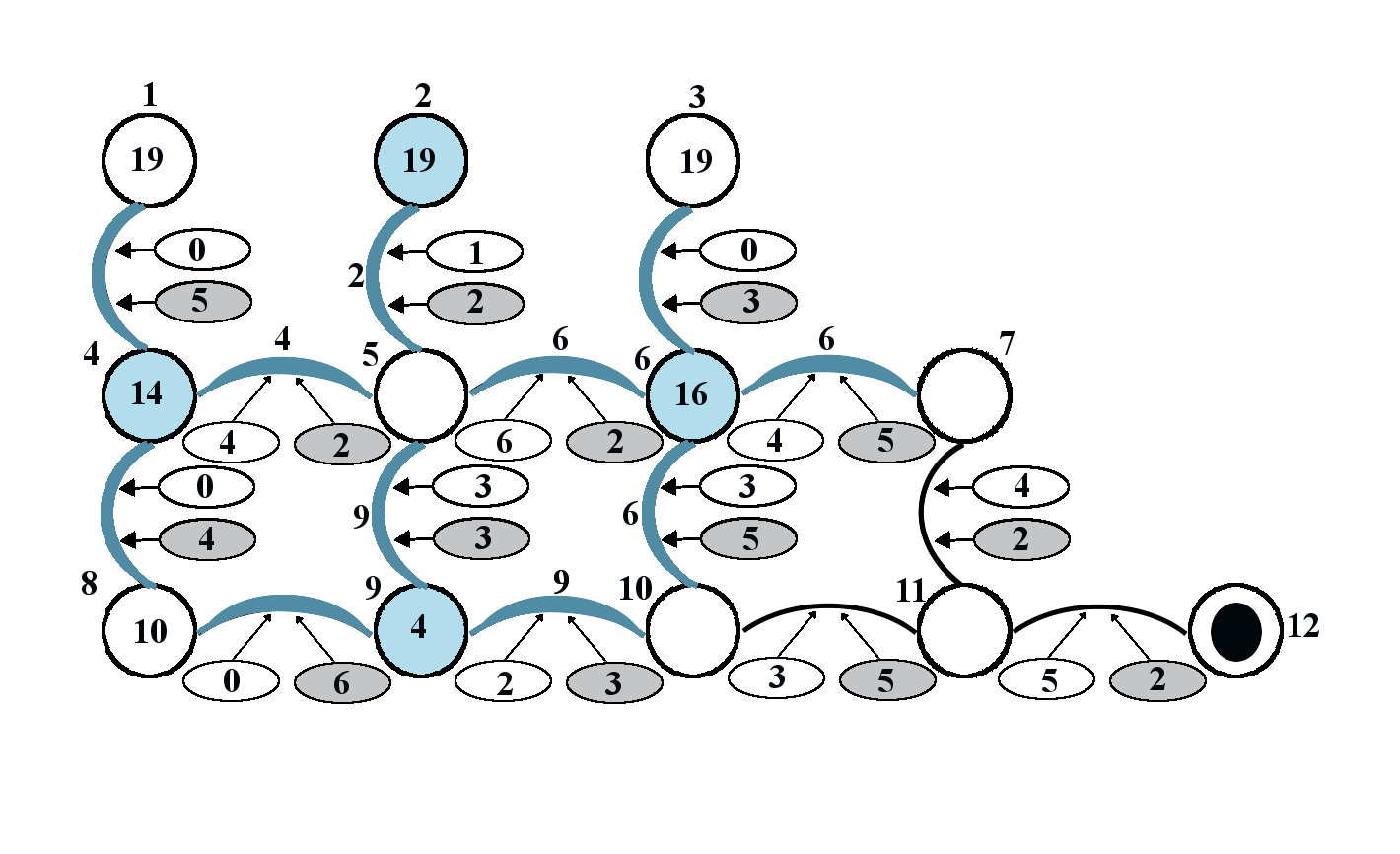}  \hfill \includegraphics[scale=0.20]{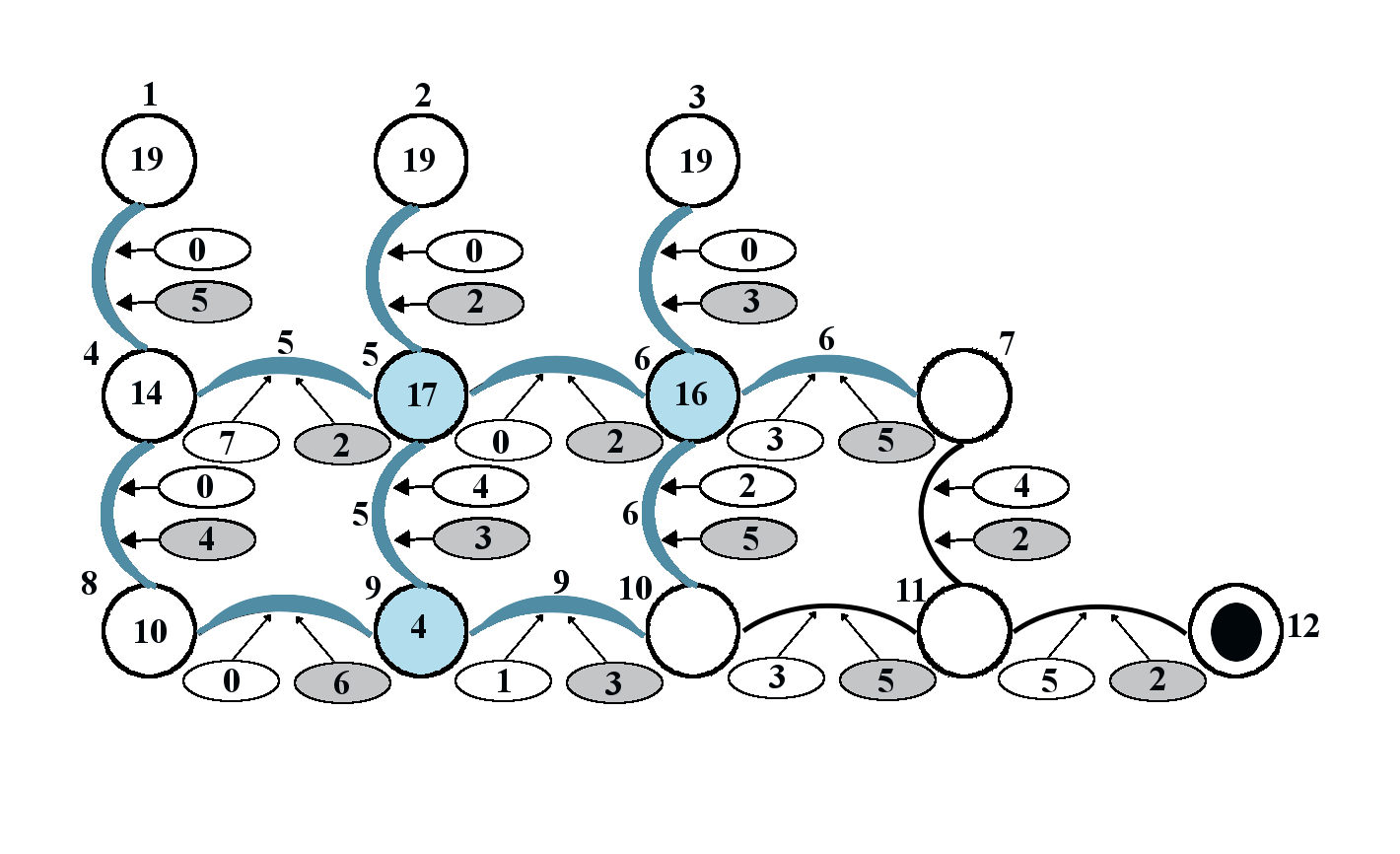} 
\end{center}
\caption{\label{fi:figure 3} The configurations after the seconds 5 and 6.}
\end{figure}

\begin{figure}[h]\centering
\begin{center}
\includegraphics[scale=0.20]{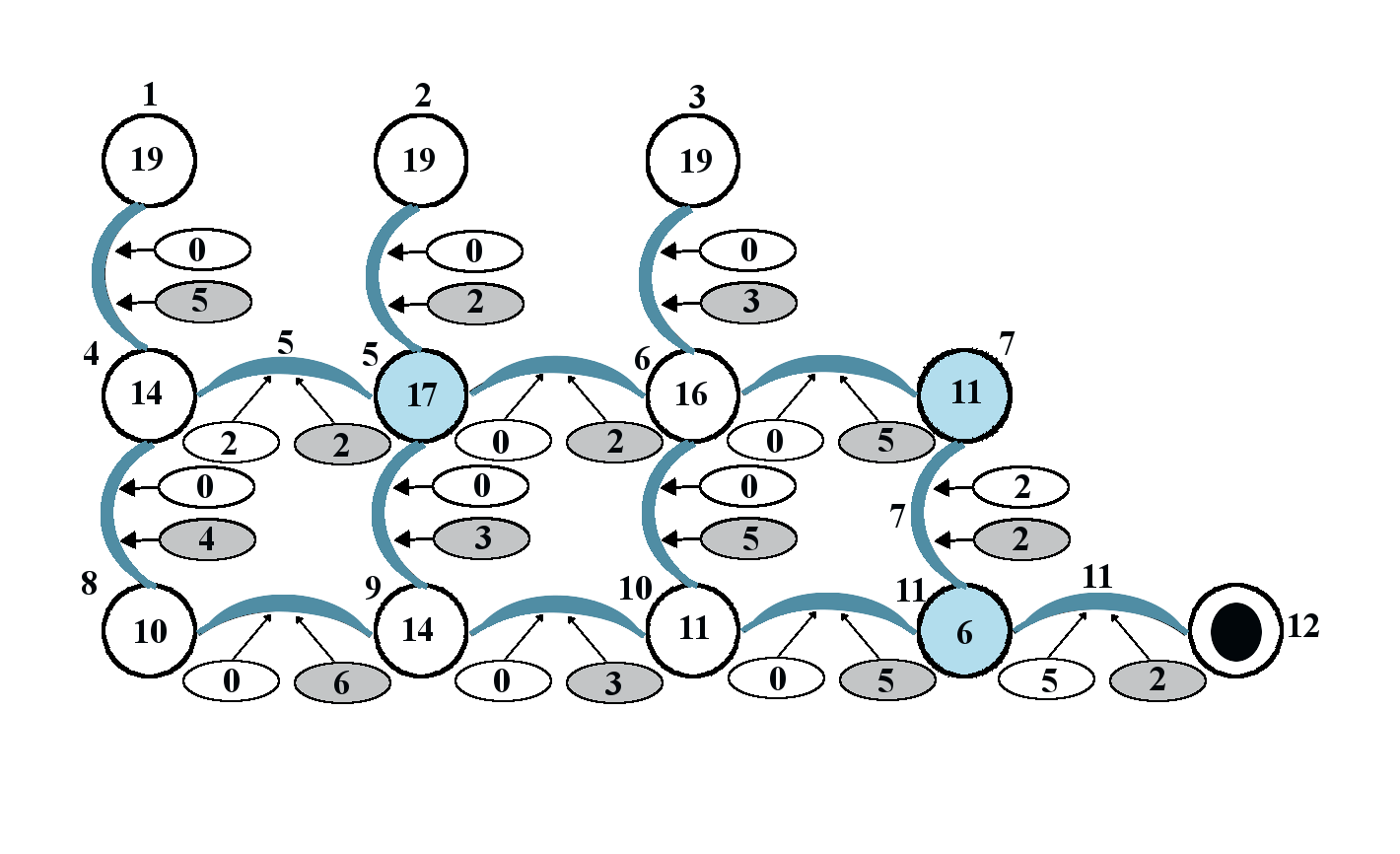}  \hfill \includegraphics[scale=0.20]{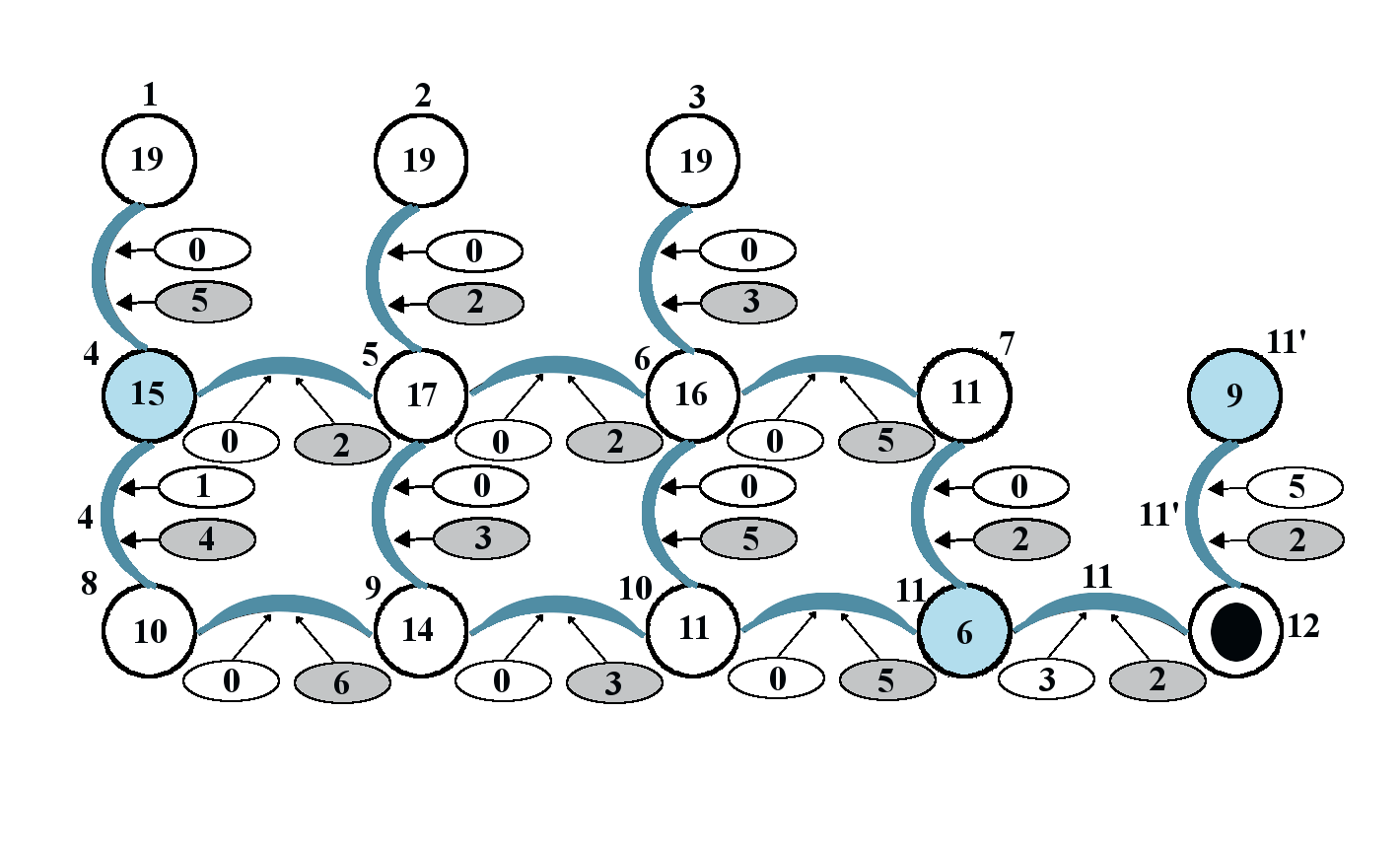} 
\end{center}
\caption{\label{fi:figure 4} The configurations after the seconds 11 and 13.}
\end{figure}

A special care has to be taken when a water flow reaches a vertex that is already active. In the case of the graph $G$ such situation happens after the $11$th second. The configuration is shown in the top picture of Figure \ref{fi:figure 4}.  
The edge $(7,11)$ has the smallest time parameter $2$. The time will progress immediately to 13 and all active edges get their time parameters decreased by $2$. In the $13$th second the water from the edge $(7,11)$ reaches the vertex $11$. The label of vertex $7$ is $\mbox{\em Label}(7)=11$, while $\mbox{\em Label}(11)=6$. The weight of the flow over the edge between these two vertices is $2$, hence this new water is of higher quality than the one present at the vertex $11$. 

In this situation we consider every active edge originating from $11$ and create a phantom edge through which this new water will flow. We will create a new vertex $11'$ with label $$\mbox{\em Label}(11')=\mbox{\em Label}(7)-f_2(7,11)=9$$ and connect it with each of the neighbors of $11$ that can get their label increased with the new flow. The only one such neighbor is $12$ and we obtain the graph as shown in the lower part of  Figure \ref{fi:figure 4}. 

It can be now easily verified that after additional $3$ seconds, i.e. in the end of the second $16$ the vertex $12$ becomes active with the label $4$. Thus we conclude that it takes water to travel $16$ seconds over the shortest path. The total weight of the shortest path is $19-4=15$. The minimizing path is $(3$, $6$, $10$, $11$, $12)$.

\section{Pseudo-code of the algorithm}
We will organize the algorithm by dividing it into smaller components. The first component is the initialization, and the others are performed in the main loop that consists of 9 major steps. The parallelizations will happen only in these individual steps of the main loop. The pseudo-code for the main function is presented in Algorithm \ref{constrainedShortestPath}. Each step will be described in full details and accompanied by a pseudo-code that outlines the main ideas. For the sake of brevity, some data structures in pseudo-code will be modeled with sets. However, the usage of sets is avoided as they cannot support insertion and deletion of elements in parallel. The sets are replaced by indicator sequences for which appropriate operations are easier to parallelize. 
For the full source code the reader is referred to \cite{matic_code}.

\begin{algorithm}
\caption{Main function}\label{constrainedShortestPath}
{\bf Input: } Graph $G=(V,E)$;  $A, B\subset G$ such that $A\cap B=\emptyset$, $M\in\mathbb R$, two functions $f_1, f_2:E\to\mathbb R$.

$\quad\quad\quad$ $f_1(e)$ is the time to travel over the edge $e$ and $f_2(e)$ is the weight of the edge $e$. 

{\bf Output: } The shortest time to travel from $A$ to $B$ over a path whose weight is less than $M$.

\begin{algorithmic}[1]
\Function{main}{}
\State Initialization
\State $L=$ShortestTravelTimeAndTerminalConditionCheck
\While{$L=0$}
\State TriggerVertices
\State AnalyzeTriggeredVertices
\State GetInputFromPhantoms
\State TriggerEdges
\State TreatTriggeredEdges
\State $L=$ShortestTravelTimeAndTerminalConditionCheck
\State FinalTreatmentOfPhantoms
\State FinalTreatmentOfVertices
\State FinalTreatmentOfActiveEdges
\EndWhile
\State\Return L
\EndFunction
\end{algorithmic}
\end{algorithm}

\section{Memory management and initialization}
\subsection{Labels for vertices and edges}
\noindent  In this section we will describe the memory management of variables necessary for the implementation of the algorithm. Before providing the precise set of variables let us describe the information that has to be carried throughout the execution process. As we have seen before, the vertices will have labels assigned to them. %Denote by $\mbox{\em Label}(P)$ the label of the vertex $P$.  
Initially we label each vertex of $G$ with $0$ except for vertices in $A$ which are labeled by $M$. 

To each vertex and edge in $G$ we assign a {\em State}. The vertices have states in the set $\{${\em active}, {\em inactive}$\}$. Initially all vertices in $A$ are {\em active}, while the other vertices are {\em inactive}. 
The states of the edges belong to the set $\{${\em active}, {\em passive}, {\em used}, {\em just used}$\}$.
Initially the edges adjacent to the vertices in $A$ are set to active while all other are passive.

To each edge we associate a pointer to one of its endpoints and call it  {\em Source}. This variable is used at times when the water is flowing through the edge and it records the source of the current water flow. Initially, to each edge that originates from a vertex in $A$ we set the source to be the pointer to the vertex in $A$. All other edges have their source initially set to $0$.

There is additional variable  $\mbox{\em Time}$ that represents the time and is initially set to $0$. %We will also need two sets of vertices called {\em TriggeredVertices} and {\em PhantomVertices} that are initialized as empty sets. 

\subsection{Termination}
The algorithm terminates if one of the following two conditions is satisfied:
\begin{enumerate}
\item[$1^{\circ}$] A vertex from $B$ becomes {\em active}. The variable $\mbox{\em Time}$ contains the time it takes to reach this vertex along the shortest path $\hat \pi$, i.e. $$\mbox{\em Time}=F_1\left(\hat\pi\right).$$ The label of the last vertex $\hat B$ on the path allows us to determine the value $F_2\left(\hat \pi\right)$. Namely, $$F_2 \left(\hat \pi\right)=M-\mbox{Label}(\hat B).$$ 

We will not go into details on how to recover the exact shortest path. Instead we will just outline how this can be done. We need to identify the {\em used} edge $f$ (or one of the used edges, if there are more than one) that is adjacent to $\hat B$. This edge can help us in finding the second to last point of the path $\hat \pi$. Let us denote by $F$  the other endpoint of $f$. It could happen that $F$ is a phantom vertex (i.e. a copy of another vertex), and we first check whether $F\in \mbox{\em PhantomVertices}$. If this is not the case, then $F$ is the second to last element of the path $\hat \pi$. If $F\in\mbox{\em PhantomVertices}$ then the vertex $F$ is a copy of some other vertex in the graph and the phantom vertex $F$ has the pointer to the original based on which it is created. This original vertex is the second to last point on the path $\hat\pi$.

\item[$2^{\circ}$] There is no {\em active} edge in the graph. In this case there is no path that satisfies the constraint $F_2\leq M$. 
\end{enumerate} 

\begin{algorithm}\caption{Function that checks whether the algorithm has finished and returns the time for travel over the shortest path}
\label{terminalCondition}
\begin{algorithmic}[1]
\Function{ShortestPathLengthAndTerminalConditionCheck}{}
\State // Returns $0$ if the path is not found yet.
\State // Returns $-1$ if there is no path with weight smaller than $M$.
\State // Returns the weight of the shortest path if it is found.
\State // A non-zero return value is the indication that the algorithm is over.
\State $\#${\bf Performed in parallel}
\If{ $\exists B_0\in B$ such that $B_0=\mbox {\em active}$}
\State $\mbox{\em result}\gets \mbox{\em Time}$
\Else 
\If{there are no active vertices}
\State  $\mbox{\em result}\gets -1$
\Else 
\State $\mbox{\em result}\gets 0$
\EndIf
\EndIf
\State $\#${\bf barrier}
\State \Return $\mbox{\em result}$
\EndFunction
\end{algorithmic}
\end{algorithm}

\subsection{Sequences accessible to all processing elements} 
It is convenient to store the vertices and edges in sequences accessible to all processing elements. We will assume here that the degree of each vertex is bounded above by $d$. 

\subsubsection{Vertices} Each vertex takes $5$ integers in the sequence of vertices. The first four are name, label, status, and the location of the first edge in the sequence of edges. The fifth element is be used to store a temporary replacement label. Initially, and between algorithm steps, this label is set to $-1$.

When a first drop of water reaches an   inactive vertex $V$, we say that the vertex is {\em triggered}, and that state exists only temporarily during an algorithm cycle. In the end of the algorithm cycle some triggered vertices become active. However it could happen that a triggered vertex does not get a water flow of higher quality than the one already present at the vertex. The particular triggered vertex with this property does not get activated.

\subsubsection{Edges} Each edge $e$ takes $8$ integers in the sequence of edges. Although the graph is undirected, each edge is stored twice in the memory. 
The $8$ integers are the start point, the end point, remaining time for water to travel over the edge (if the edge is active), the weight of the travel $f_2(e)$, the initial passage time $f_1(e)$,  the label of the vertex that is the source of the current flow through the edge (if there is a flow), status, and the location of the same edge in the opposite direction.

\subsubsection{Active vertices} 
The sequence contains the locations of the vertices that are active. This sequence removes the need of going over all vertices in every algorithm step. The locations are sorted in decreasing order. In the end of the sequence we will add triggered vertices that will be joined to the active vertices in the end of the cycle. 

\subsubsection{Active edges}
The role of the sequence is similar to the one of active vertices. The sequence maintains the location of the active edges. Each edge is represented twice in this sequence. The second appearance is the one in which the endpoints are reversed. The locations are sorted in decreasing order. During the algorithm cycle we will append the sequence with triggered edges. In the end of each cycle the triggered edges will be merged to the main sequence of active edges.

\subsubsection{Sequence of phantom edges} The phantom edges appear when an active vertex is triggered with a new drop of water. Since the vertex is active we cannot relabel the vertex. Instead each of the edges going from this active triggered vertex need to be doubled with the new source of water flowing through these new edges that are called phantoms. They will disappear once the water finishes flowing through them. 

\subsubsection{Sequence of elements in $B$} Elements in $B$ have to be easily accessible for quick check whether the algorithm has finished. For this reason the sequence should be in global memory.

Listing \ref{initializationProcedure} summarizes the initializing procedures. 

\begin{algorithm}
\caption{Initialization procedure}
\label{initializationProcedure}
\begin{algorithmic}[1]
\Procedure{Initialization}{}
\For {$e\in E$ }
\State $\mbox{\em State}(e)\gets \mbox{\em passive}$ 
\State $\mbox{\em Source}(e)\gets 0$
\State $\mbox{\em TimeRemaining}(e)\gets 0$ 
 \EndFor
\For{$v\in V\setminus A$}
\State $\mbox{\em State}(v)\gets \mbox{\em inactive}$
\State $\mbox{\em Label}(v)\gets 0$

\EndFor

\For {$v\in A$ }
\State $\mbox{\em State}(v)\gets \mbox{\em active}$
\State $\mbox{\em Label}(v)\gets M$
\For {$e\in \mbox{\em Edges}(v)$ }
\State $\mbox{\em State}(e)\gets\mbox{\em active}$ 
\State $\mbox{\em Source}(e)\gets v$
\State $\mbox{\em TimeRemaining}(e)\gets f_1(e)$
 \EndFor
\EndFor
\State $\mbox{\em TriggeredVertices}\gets \emptyset$
\State $\mbox{\em PhantomVertices}\gets \emptyset$
\State $\mbox{\em PhantomEdges}\gets \emptyset$
\State $\mbox{\em Time}\gets 0$
\EndProcedure
\end{algorithmic}

\end{algorithm}

\section{Graph update} \noindent The algorithm updates the graph in a loop until one vertex from $B$ becomes active. Each cycle consists of the following nine steps. 

\subsection{Step 1: Initial triggering of vertices} In this step we go over all active edges and decrease their time parameters by $m$, where $m$ is the smallest remaining time of all active edges. 
If for any edge the time parameter becomes $0$, the edge becomes {\em just used} and its destination triggered.

To avoid the danger of two processing elements writing in the same location of the sequence of active vertices, we have to make sure that each processing element that runs concurrently has pre-specified location to write. This is accomplished by first specifying the number of threads in the separate variable {\em nThreads}. Whenever kernels are executed in parallel we are using only {\em nThreads} processing elements. Each processing element has its id number which is used to determine the memory location to which it is allowed to write. 
The sequence of triggered vertices has to be cleaned after each parallel execution and at that point we take an additional step to ensure we don't list any of the vertices as triggered twice.

\begin{algorithm}
\caption{Procedure TriggerVertices}
\label{step1}
\begin{algorithmic}[1]
\Procedure{TriggerVertices}{}
\State $\mbox{\em TriggeredEdges}\gets \emptyset$
\State $\#${\bf Performed in parallel}
\State $m\gets \min\left\{\mbox{\em TimeRemaining}(e): e\in \mbox{\em ActiveEdges}\right\}$ 
\State $\#${\bf barrier}
\State $\mbox{\em Time}\gets \mbox{\em Time}+m$
\State $\#${\bf Performed in parallel}
\For{$e \in \mbox{\em ActiveEdges}$}
\State $\mbox{\em TimeRemaining}(e)\gets\mbox{\em TimeRemaining}(e)-m$
\If{$\mbox{\em TimeRemaining}(e)=0$}
\State $\mbox{\em State}(e)=\mbox{\em just used}$
\State $S_e\gets \mbox{\em Source}(e)$
\State $D_e\gets \mbox{\em TheTwoEndpoints}(e)\setminus \{S_e\}$
\State $\mbox{\em TriggeredVertices}\gets \mbox{\em TriggeredVertices}\cup \{D_e\}$
\EndIf
\EndFor
\State $\#${\bf barrier}

\EndProcedure
\end{algorithmic}

\end{algorithm}

\subsection{Step 2: Analyzing triggered vertices} For each triggered vertex $Q$ we look at all of its edges that are just used. We identify the largest possible label that can result from one of just used edges that starts from $Q$. That label will be stored in the sequence of vertices at the position reserved for temporary replacement label. The vertex is labeled as just triggered. If the vertex $Q$ is not active, this label will replace the current label of the vertex in one of the later steps. If the vertex $Q$ is active, then this temporary label will be used later to construct an appropriate phantom edge. 

We are sure that different processing elements are not accessing the same vertex at the same time, because before this step we achieved the state in which there are no repetitions in the sequence of triggered vertices. 

\begin{algorithm}
\caption{Analysis of triggered vertices}
\label{step2}
\begin{algorithmic}[1] 
\Procedure{AnalyzeTriggeredVertices}{}
\State $\mbox{\em TempLabel}\gets \emptyset$
\State $\#${\bf Performed in parallel}
\For{$Q \in \mbox{\em TriggeredVertices}$}
\State $\mbox{\em TempLabel}(Q)\gets \max\left\{ \mbox{\em Label}(P)-f_2(P,Q): \mbox{\em State}(P,Q)=\mbox{\em just used}\right\}$
%\State $\mbox{\em State}(Q)\gets \mbox{\em just triggered}$
\EndFor
\State $\#${\bf barrier}

\EndProcedure
\end{algorithmic}

\end{algorithm}

\subsection{Step 3: Gathering input from phantoms} The need to have this step separated from the previous ones is the current architecture of graphic cards that creates difficulties with dynamic memory locations. It is more efficient to keep phantom edges separate from the regular edges. The task is to look for all phantom edges and decrease their time parameters. If a phantom edge gets its time parameter equal to $0$, its destination is studied to see whether it should be added to the sequence of triggered vertices. We calculate the new label that the vertex would receive through this phantom. We check whether this new label is higher than the currently known label and the temporary label from possibly previous triggering of the vertex.  The phantoms will not result in the concurrent writing to memory locations because each possible destination of a phantom could have only one edge that has time component equal to $0$.

\begin{algorithm}
\caption{Input from phantoms}
\label{step3}
\begin{algorithmic}[1] 
\Procedure{GetInputFromPhantoms}{}
\State $\#${\bf Performed in parallel}
%\For{$e \in \mbox{\em PhantomEdges}$}
%\State $\mbox{\em TimeRemaining}(e)\gets \mbox{\em TimeRemaining}(e)-m$
% \If{$\mbox{\em TimeRemaining}(e)=0$}
%\State $\mbox{\em State}(e)=\mbox{\em just used}$
%\State $S_e\gets \mbox{\em Source}(e)$
%\State $D_e\gets \mbox{\em TheTwoEndpoints}(e)\setminus \{S_e\}$
%\State $\mbox{\em TriggeredVertices}\gets \mbox{\em TriggeredVertices}\cup \{D_e\}$
%\EndIf
%\EndFor
\State Decrease time parameters of fantom edges (as in Listing \ref{step1})
\State Trigger the destinations of phantom edges (as in Listing \ref{step1})
\State $\#${\bf barrier}
\State $\#${\bf Performed in parallel}
\State Analyze newly triggered vertices, in a way similar to Listing \ref{step2}
\State $\#${\bf barrier}

\EndProcedure
\end{algorithmic}

\end{algorithm}

\subsection{Step 4: Triggering edges} 
In this step we will analyze the triggered vertices and see whether each of their neighboring edges needs to change the state. Triggered vertices are analyzed using separate processing elements. A processing element analyzes the vertex $Q$ in the following way. 

Each edge $j$ of $Q$ 
will be considered triggered if it can cause the other endpoint to get better label in future through $Q$. The edge $j$ is placed in the end of the sequence of active edges.

\begin{algorithm}
\caption{Procedure that triggers the edges}
\label{step4}
\begin{algorithmic}[1] 
\Procedure{TriggerEdges}{}
\State $\#${\bf Performed in parallel}
\For{$Q \in \mbox{\em TriggeredVertices}$}
\For{$P \in \mbox{\em Neighbors}(Q)$}
\If {$\mbox{\em TempLabel}(Q)-f_2(P,Q)>\mbox{\em Label}(P)$}
\State $\mbox{\em State}(P,Q)\gets\mbox{\em active}$
\State $\mbox{\em TriggeredEdges}\gets \mbox{\em TriggeredEdges}\cup\{(P,Q)\}$
\EndIf
\EndFor
\EndFor

\State $\#${\bf barrier}

\EndProcedure
\end{algorithmic}

\end{algorithm}

\subsection{Step 5: Treatment of triggered edges} 
Consider a triggered edge $j$. We first identify its two endpoints. For the purposes of this step we will identify the endpoint with the larger label, call it the source, and denote by $S$. The other will be called the destination and denoted by $D$. In the end of the cycle, this vertex $S$ will become the source of the flow through $j$.

Notice that at least one of the endpoints is triggered. If only one endpoint is triggered, then we are sure that this triggered endpoint is the one that we designated as the source $S$. 

We then look whether the source $S$ was active or inactive before it was triggered.

\subsubsection{Case in which the source $S$ was inactive before triggering} 
 There are several cases based on the prior status of $j$. If $j$ was passive, then it should become active and no further analysis is necessary. If it was used or just used, then it should become active and the time component should be restored to the original one. Assume now that the edge $j$ was active. Based on the knowledge that $S$ was inactive vertex we can conclude that the source of $j$ was $D$. However we know that the source of $j$ should be $S$ and hence the time component of $j$ should be restored to the backup value. 

Consequently, in the case that $S$ was inactive, regardless of what the status of $j$ was, we are sure its new status must be active and its time component can be restored to the original value. This restoration is not necessary in the case that $j$ was passive, although there is no harm in doing it. 

If the edge $j$ was not active before, then the edge $j$ should be added to the list of active edges. 
If the edge $j$ was active before, then it should be removed from the list of triggered edges because all triggered edges will be merged into active edges. The edge $j$ already appears in the list of active edges and need not be added again. 

\subsubsection{Case in which the source $S$ was active before triggering} In this case we create phantom edges. Each such triggered edge generates four entries in the phantom sequence. The first one is the source, the second is the destination, the third is the 
label of the source (or the label stored in the temporary label slot, if higher), and the fourth is the original passage time through the edge $j$.

\begin{algorithm}
\caption{Treatment of triggered edges}
\label{step5}
\begin{algorithmic}[1] 
\Procedure{TreatTriggeredEdges}{}
\State $\#${\bf Performed in parallel}
\For{$j \in \mbox{\em TriggeredEdges}$}
\State $S\gets \mbox{The endpoint of }j\mbox{ with larger label}$
\State $D\gets \mbox{The endpoint of }j\mbox{ with smaller label}$
\State $\mbox{\em OldStateOfS}\gets \mbox{\em State}(S)$
\State $\mbox{\em OldStateOfJ}\gets \mbox{\em State}(j)$
\If{$\mbox{\em OldStateOfS}=\mbox{\em inactive}$}
 
\State $\mbox{\em State}(j)\gets\mbox{\em active}$
\State $\mbox{\em Source}(j)\gets S$
\State $\mbox{\em TimeRemaining}(j)\gets f_1(j)$

\EndIf
\If{$\mbox{\em OldStateOfS}=\mbox{\em active}$}
\State Create a phantom vertex $S'$ and connect it to $D$
\State $\mbox{\em TimeRemaining}(S',D)\gets f_1(S,D)$
\EndIf
\EndFor

\State $\#${\bf barrier}

\EndProcedure
\end{algorithmic}

\end{algorithm}

\subsection{Step 6: Checking terminal conditions} In this step we take a look whether a vertex from $B$ became active or if there are no active edges. These would be the indications of the completion of the algorithm. The function that checks the terminal conditions is presented earlier in Listing \ref{terminalCondition}.

\subsection{Step 7: Final treatment of phantoms} In this step we go once again over the sequence of phantoms and remove each one that has its time parameter equal to $0$.

%\State FinalTreatmentOfVertices
%\State FinalTreatmentOfActiveEdges

\begin{algorithm}
\caption{Final treatment of phantoms}
\label{step7}
\begin{algorithmic}[1] 
\Procedure{FinalTreatmentOfPhantoms}{}
\State $\#${\bf Performed in parallel}
\For{$j \in \mbox{\em PhantomEdges}$}
\If{$\mbox{\em TimeRemaining}(j)=0$}
\State Remove $j$ and its source from the sequence of phantoms

\EndIf
\EndFor

\State $\#${\bf barrier}

\EndProcedure
\end{algorithmic}

\end{algorithm}

\subsection{Step 8: Final treatment of vertices} In this step of the program the sequence of active vertices is updated so it contains new active vertices and looses the vertices that may cease to be active.

\subsubsection{Preparation of triggered vertices} For each triggered vertex $Q$ we first check whether it was inactive before. If it was inactive then its label becomes equal to the label stored at the temporary storing location in the sequence of vertices. If it was active, its label remains unchanged. The phantoms were created and their labels are keeping track of the improved water quality that has reached the vertex $Q$. 

We may now clean the temporary storing location in the sequence of vertices so it now contains the symbol for emptiness (some pre-define negative number).

\subsubsection{Merging triggered with active vertices} 
Triggered vertices are now merged to the sequence of active vertices. 

\subsubsection{Check active vertices for potential loss of activity} For each active vertex $Q$ look at all edges from $Q$. If there is no active edge whose source is $Q$, then $Q$ should not be active any longer.

\subsubsection{Condensing the sequence of active vertices} After previous few steps some vertices may stop being active in which case they should be removed from the sequence.

%\State FinalTreatmentOfActiveEdges

\begin{algorithm}
\caption{Final treatment of vertices}
\label{step8}
\begin{algorithmic}[1] 
\Procedure{FinalTreatmentOfVertices}{}
\State $\#${\bf Performed in parallel}
\For{$Q\in\mbox{\em TriggeredVertices}$}
\If{$\mbox{\em State}(Q)=\mbox{\em inactive}$}
\State $\mbox{\em State}(Q)\gets \mbox{\em active}$
\State $\mbox{\em Label}(Q)\gets \mbox{\em TempLabel}(Q)$
\EndIf
\EndFor

\State $\#${\bf barrier}
\State $\mbox{\em TempLabel}\gets \emptyset$
\State $\#${\bf Performed in parallel}
\State Merge triggered vertices to active vertices
\State $\#${\bf barrier}
\State $\#${\bf Performed in parallel}
\For{$Q\in\mbox{\em TriggeredVertices}$}
\If{there are no active edges starting from $Q$}
\State $\mbox{\em State}(Q)\gets\mbox{\em inactive}$
\EndIf
\EndFor
\State $\#${\bf barrier}
\EndProcedure
\end{algorithmic}

\end{algorithm}

\subsection{Step 9: Final treatment of active edges} 
We first need to merge the triggered edges with active edges. Then all just used edges have to become used and their source has to be re-set so it is not equal to any of the endpoints. Those used edges should be removed from the sequence of  active edges.

The remaining final step is to condense the obtained sequence so there are no used edges in the sequence of active edges.

\begin{algorithm}
\caption{Final treatment of active edges}
\label{step9}
\begin{algorithmic}[1] 
\Procedure{FinalTreatmentOfActiveEdges}{}
 
\State $\#${\bf Performed in parallel}
\State Merge triggered edges to active edges
\State $\#${\bf barrier}
\State $\#${\bf Performed in parallel}
\State Transform all {\em just used} into {\em used} and erase their {\em Source} components
\State $\#${\bf barrier}

\EndProcedure
\end{algorithmic}

\end{algorithm}

%\begin{algorithmic}[1]
%\Procedure{Initialization2}{}
%\State $\mbox{\em stringlen} \gets \mbox{length of }\mbox{\em string}$
%\State $i \gets \mbox{\em patlen}$
%\State $\mbox{\em top}$:
%\If {$i > \mbox{\em stringlen}$} \Return false
%\EndIf
%\State $j \gets \mbox{\em patlen}$
%\State $\mbox{\em loop}$:
%\If {$\mbox{\em string}(i) = \mbox{\em path}(j)$}
%\State $j \gets j-1$.
%\State $i \gets i-1$.
%\State \mbox{\bf goto}$ \mbox{\em loop}.$
%\State \mbox{\bf close};
%\EndIf
%\State $i \gets i+\max(\textit{delta}_1(\textit{string}(i)),\textit{delta}_2(j))$.
%\State \textbf{goto} \emph{top}.
%\EndProcedure
%\end{algorithmic}

\section{Large sets of active vertices}
\noindent In this section we will prove that it is possible for the set of active vertices in dimension $2$ to contain more than $O(n)$ elements. We will construct examples in the case when the time to travel over each vertex is from the set $\{1,2\}$ and when $M=+\infty$.  

We will consider the subgraph $V_n= [-n,n]\times[0,n]$ of $\mathbb Z^2$. At time $0$ the water is located in all vertices of the $x$ axis. For sufficiently large $n$ we will provide an example of configuration $\omega$ of passage times for the edges of the graph $V_n$ such that the number of active vertices at time $n$ is of order $n\log n$. This would establish a lower bound on the probability that the number of active vertices at time $t$ is large. 

Let us assume that each edge of the graph has the time component assigned from the set $\{1,2\}$ independently from each other. Assume that the probability that $1$ is assigned to each edge is equal to $p$, where $0<p<1$.

\begin{theorem} \label{lower_bound} There exists $t_0\geq 0$, $\mu>0$, and $\alpha>0$ such that for each $t>t_0$ there exists $n$ such that the number $A_t$ of active vertices at time $t$ in the graph $V_n$ satisfies \begin{eqnarray*}\mathbb P\left(A_t\geq \alpha t\log t\right)\geq e^{-\mu t^2}.\end{eqnarray*}
\end{theorem}

To prepare for the proof of the theorem we first study the evolution of the set of active edges in a special case of a graph. Then we will construct a more complicated graph where the set of active edges will form a fractal of length $t\log t$.  

\begin{lemma} If all edges on the $y$-axis have time parameter equal to $1$ and all other edges have their time parameter equal to $2$, then at time $T$ the set of active vertices is given by \begin{eqnarray*}A_T&=&\left\{(0,T)
\right\} \cup\left\{(0,T-1)
\right\} \cup \bigcup_{k=1}^{\left\lfloor \frac{T+1}4\right\rfloor} \left\{\left( -k,T-2k \right),
\left( k,T-2k \right)
\right\} \\ &&\cup \bigcup_{z\in \mathbb Z\setminus \left\{-\left\lfloor\frac{T+1}{4}\right\rfloor, \dots, \left\lfloor \frac{T+1}{4}\right\rfloor\right\}}\left\{\left(z,\left\lfloor \frac T2\right\rfloor\right)\right\}.
\end{eqnarray*}
\end{lemma}
\begin{proof} 
After $T-2k$ units of time the water can travel over the path $\gamma_k$ that consists of vertices $(0,0)$, $(0,1)$, $\dots$, $(0,T-2k)$. In additional $2k$ units of time the water travels over the path $\gamma'_k$ that consists of vertices $(0,T-2k)$, $(1,T-2k)$, $\dots$, $(k,T-2k)$.  
\begin{figure}
\centering
\begin{center}
\includegraphics[scale=0.25]{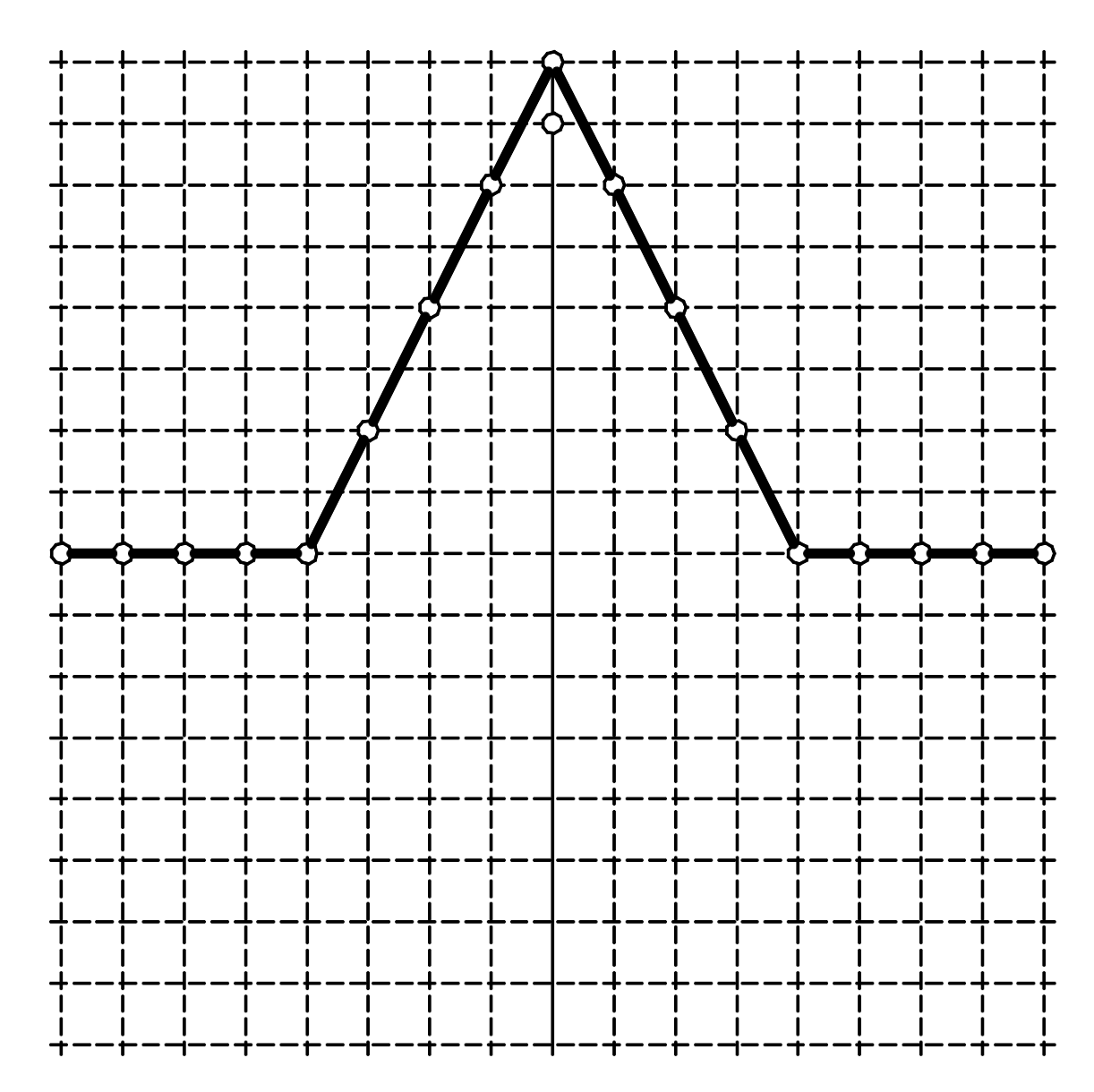}  \end{center}
\caption{\label{fi:figure 5} The active edges at time $T$.}
\end{figure} 
Consider any other path that goes from $x$ axis to the point $(k,T-2k)$ for some fixed $k\leq \left\lfloor \frac{T+1}4\right\rfloor$. If the path takes some steps over edges that belong to $y$ axis then it would have to go over at least $k$ horizontal edges to reach $y$ axis, which would take $2k$ units of time. The path would have to take at least $T-2k$ vertical edges, which would take at least $T-2k$ units of time. Thus the travel would be longer than or equal to $T$. 

However, if the path does not take steps over the edges along $y$ axis then it would have to take at least $T-2k$ steps over edges that have passage time equal to $2$. This would take $2(T-2k)=2T-4k$ units of time. If $T+1$ is not divisible by $4$, then $k<\frac{T+1}4$ and 
\begin{eqnarray*}2T-4k>2T-T-1=T-1,\end{eqnarray*} which would mean that the travel time is at least $T$. If $T+1$ is divisible by $4$ and $k= \left\lfloor \frac{T+1}4\right\rfloor$ then the vertical path would reach $(k,T-2k)$ at time $T-1$. However, the vertex $(k,T-2k)$ would still be active because the water would not reach $(k+1,T-2k)$ which is a neighbor of $(k,T-2k)$. 
\end{proof}

Let us denote by $N_t$ the number of active vertices at time $t$ whose $x$ coordinate is between $-t$ and $t$,
\begin{eqnarray*}N_t=\left\{ (x,y)\in \{-t, -t+1, \dots, t-1, t\}\times \mathbb Z_0^+: (x,y) \mbox{ is active at time } t\right\}.\end{eqnarray*}

\begin{theorem}  \label{lower_bound_construction}
There exist real numbers $\alpha$ and $t\geq 0$ and an environment $\omega$ for which \begin{eqnarray*}N_t(\omega)\geq \alpha t\log t.\end{eqnarray*}
\end{theorem}
\begin{proof} Assume that $t=2^k$ for some $k\in \mathbb N$. Let  us define the following points with their coordinates $T=(0,t)$, $L=\left(-\frac{t}2,0\right)$, and $O=\left(0,\frac t2\right)$. 
We will recursively construct the sequence of 
pairs $\left(\omega_1,\mathcal I_1\right)$, $\left(\omega_2,\mathcal I_2\right)$,  $\dots$, $\left(\omega_k,\mathcal I_k\right)$ where $\omega_j$ is an assignment of passage times to the edges and $\mathcal I_j$ is a subgraph of $\mathbb Z^2$. This subgraph will be modified recursively. All edges in $\mathcal I_j$ have passage times equal to $2$ in the assignment $\omega_j$.  Having defined the pair $\left(\omega_j,\mathcal I_j\right)$ we will improve passage times over some edges in the set $\mathcal I_j$ by changing them from $2$ to $1$. This way we will obtain a new environment $\omega_{j+1}$ and we will define a new set $\mathcal I_{j+1}$ to be a subset of $\mathcal I_j$. The new environment   $\omega_{j+1}$ will satisfy \begin{eqnarray*}N_t(\omega_{j+1})\geq N_t(\omega_j)+\beta t,\end{eqnarray*} for some $\beta>0$.

Let us first construct the pair $\left(\omega_1,\mathcal I_1\right)$. We will only construct the configuration to the left of the $y$ axis and then reflect it across the $y$ axis to obtain the remaining configuration. 

All edges on the $y$ axis have the passage times equal to $1$, and all edges on the segment $LO$ have the passage times equal to $1$. All other edges have the passage times equal to $2$. Define $\mathcal I_1=\triangle LOT$. 
Then the polygonal line $LYT$ contains the active vertices whose $x$ coordinate is between $-t$ and $0$.

The environment $\omega_2$ is constructed in the following way. Let us denote by $L_0$ and $T_0$  the midpoints of $LO$ and $TO$. Let $X$ be the midpoint of $LT$. We change all vertices on $L_0X$ and $T_0X$ to have the passage time equal to $1$. We define $\mathcal I_2= \triangle LL_0X\cup \triangle XT_0T$.   

Let $L_1$ and $L_2$ be the midpoints of $LL_0$ and $L_0O$ and let $L'$ and $L''$ be the intersections of $XL_1$ and $XL_2$ with $LY$. The points $T'$ and $T''$ are defined in an analogous way: first $T_1$ and $T_2$ are defined to be the midpoints of $TT_0$ and $OT_0$ and $T'$ and $T''$ are the intersections of $XT_1$ and $XT_2$ with $TY$. 

The polygonal line $LL'XL''YT''XT'T$ is the set of active edges that are inside the triangle $LOT$. The following lemma will allow us to calculate $N_t\left(\omega_2\right)-N_t\left(\omega_1\right)$.

\begin{figure}
\centering
\begin{center}
\includegraphics[scale=0.25]{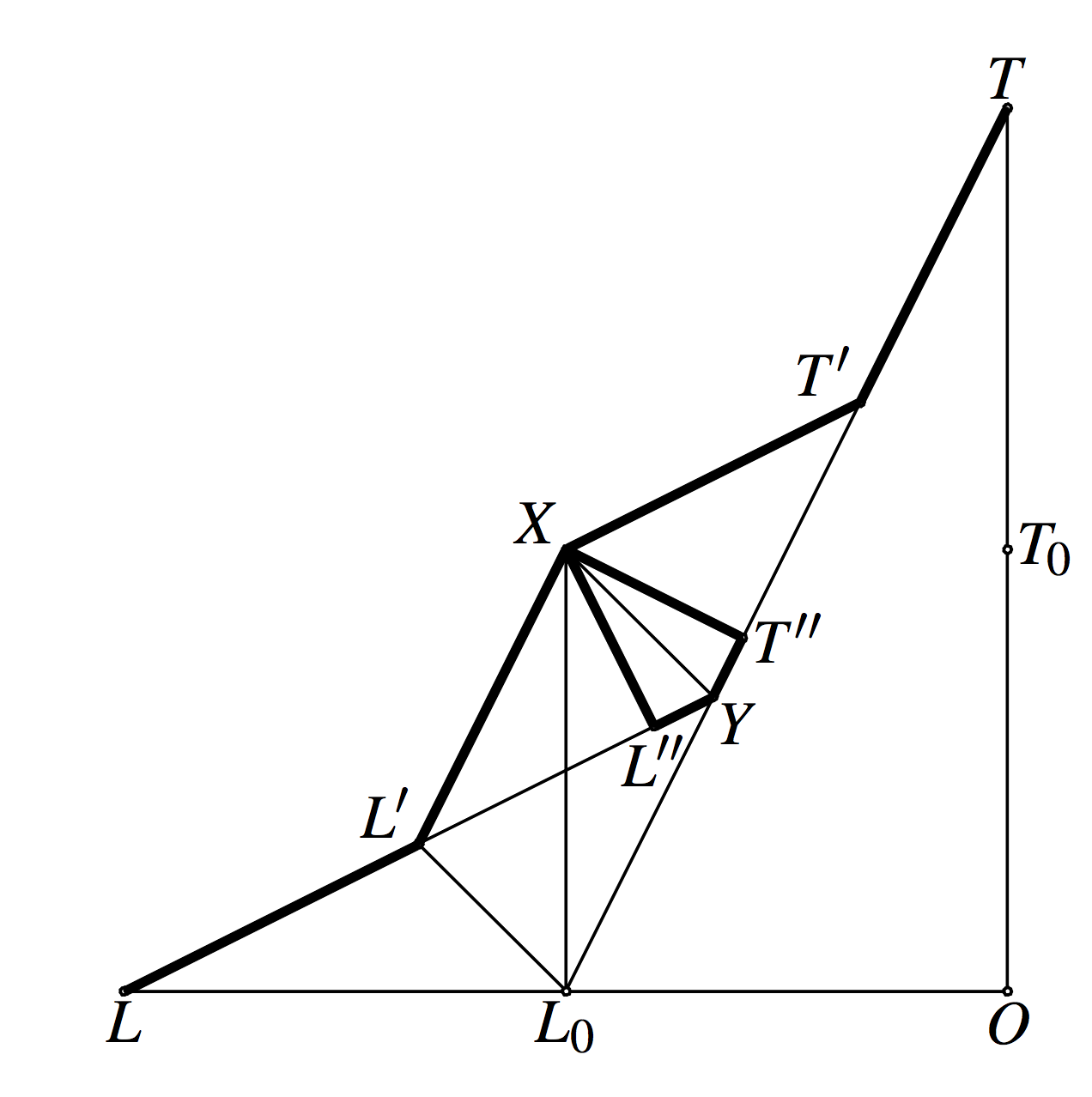}  
\end{center}
\caption{\label{fi:figure 6} The set of active edges in configuration $\omega_2$.}
\end{figure}
\begin{lemma}  \label{lemma_geometry}
Let $\Lambda$ and $\lambda$ denote the lengths of the polygonal lines
$LL'XL''YT''XT'T$ and $LYT$ respectively. If $t$ is the length of $OT$ then 
\begin{eqnarray*}\Lambda = \lambda+ \frac{4}{3\sqrt 5} t.\end{eqnarray*}
\end{lemma}

\begin{proof} It suffices to prove that
$LL'+L'X+XL''+L''Y = LY+ \frac{2}{3\sqrt 5}t$. 
From the similarities $\triangle LL_0X\sim\triangle LOT$ and $\triangle LL_0L'\sim LOX$ we have that $L_0L'\| OX$. Therefore $L'$ is the midpoint of $LY$ and $LY=LL'+L'Y=LL'+L'X$. It remains to prove that
$XL''+L''Y=\frac{2}{3\sqrt 5}t$. 
From \begin{eqnarray*}\angle L_0XL''=\angle L'XL_0=\angle L_0LL''\end{eqnarray*} we conclude that the quadrilateral $LL_0L''X$ is inscribed in a circle. The segment $LX$ is a diameter of the circle hence
$\angle LL''X=\angle LL_0X=90^{\circ}$. 
%\begin{eqnarray*}\angle LL''X=\angle LL_0X=90^{\circ}.\end{eqnarray*}  
We also have $\angle L''XY=\angle L_0XY-\angle L_0XL''=45^{\circ}-\angle OLT_0=
45^{\circ}-\mbox{arctan}\frac12$.
  
The point $Y$ is the centroid of $\triangle LOT$ hence $XY=\frac13 XO=\frac1{3\sqrt 2}t$.
Therefore \begin{eqnarray*}XL''+L''Y&=&XY\cos\left(45^{\circ}-\mbox{arctan}\frac12\right)+XY\sin\left(45^{\circ}-\mbox{arctan}\frac12\right)\\&=&
\frac{\cos\left(45^{\circ}-\mbox{arctan}\frac12\right)+\sin\left(45^{\circ}-\mbox{arctan}\frac12\right)}{3\sqrt 2}t\\
&=&
\frac{\cos\left(45^{\circ}-\mbox{arctan}\frac12\right)\cos 45^{\circ}+\sin\left(45^{\circ}-\mbox{arctan}\frac12\right)
\sin 45^{\circ}}{3}t\\
&=&\frac{\cos\left(45^{\circ}-\mbox{arctan}\frac12-45^{\circ}\right)}{3}t
=\frac{\cos\left( \mbox{arctan}\frac12\right)}{3}t\\&=&\frac{2}{3\sqrt 5}t.
\end{eqnarray*}
\end{proof}
The number of edges on each of the segments of the polygonal lines we obtained is equal to $\frac{u}{\sqrt 5}$, where $u$ is the length of the segment. Using this fact with the previous lemma applied to  both $\triangle LOT$ and its reflection along $OT$ gives us 
\begin{eqnarray*}N_t\left(\omega_2\right)-N_t\left(\omega_1\right)=\frac{4}{3\sqrt 5}t\cdot \frac1{\sqrt5}=\frac{4}{15}t.\end{eqnarray*}
\begin{figure}
\centering
\begin{center}
\includegraphics[scale=0.25]{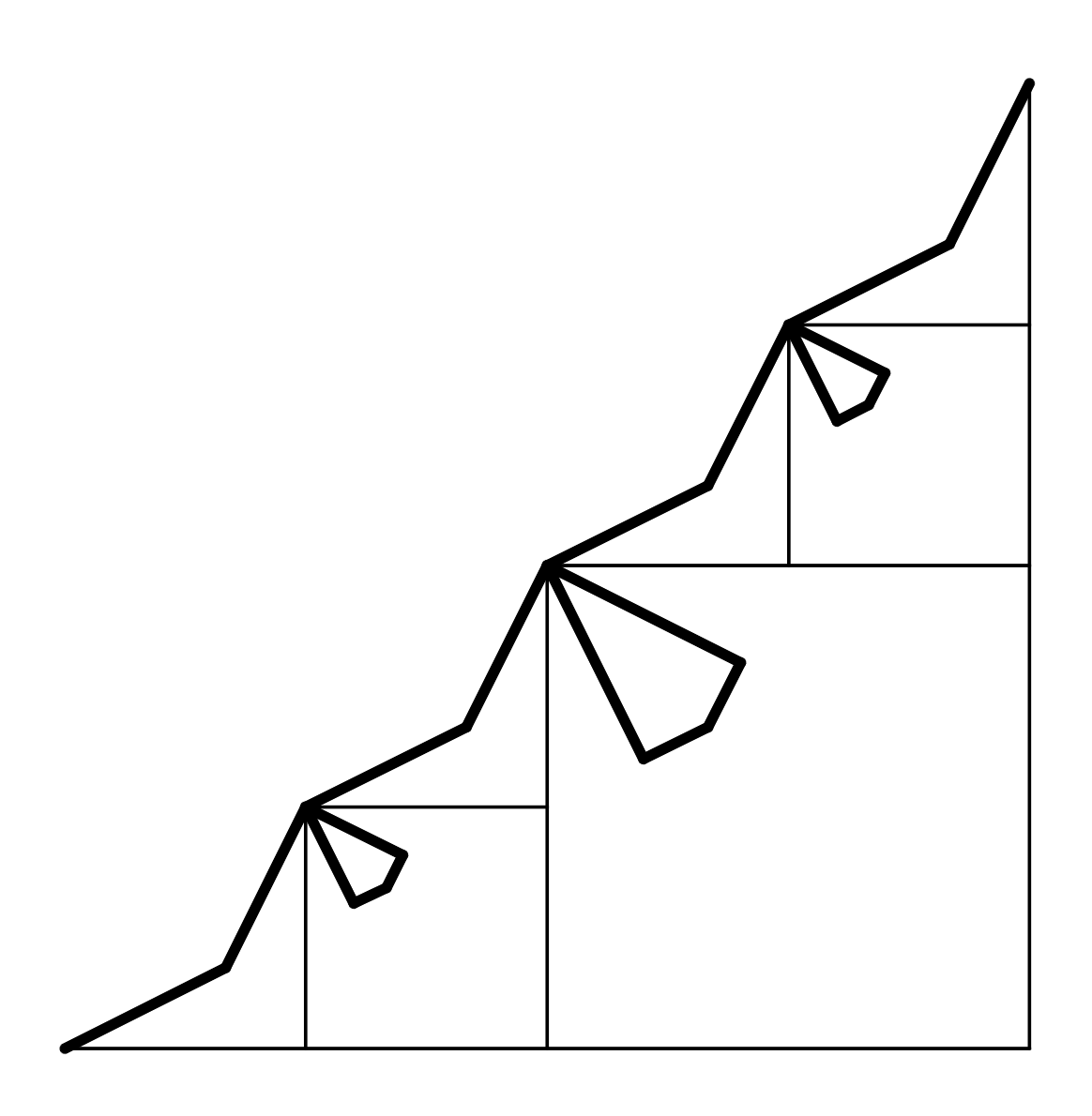}   \end{center}
\caption{\label{fi:figure 7} The set of active edges in configuration $\omega_3$.}
\end{figure} 
We now continue in the same way and in each of the triangles $LL_0X$ and $XT_0T$ we perform the same operation to obtain $\omega_3$ and $\mathcal I_3$. Since the side length of $LL_0X$ is $\frac t2$, the increase in the number of elements in the new set of active vertices is $\frac{2}{15}\cdot \frac t2$. However, this number has to be now multiplied by $4$ because there are $4$ triangles to which the lemma is applied: $\triangle LL_0X$, $\triangle XT_0T$, and the reflections of these two triangles with respect to $OT$. Therefore the increase in the number of active vertices is $N_t\left(\omega_3\right)-N_t\left(\omega_2\right)=4\cdot \frac{2}{15}\cdot \frac t2=\frac{4}{15}t$.

This operation can be repeated $k$ times and we finally get that \begin{eqnarray*}N_t\left(\omega_k\right)=N_t\left(\omega_1\right)+(k-1)\cdot 
\frac{4}{15}t\geq k\cdot \frac{4}{15}t.\end{eqnarray*} Thus the theorem holds if we set $\alpha= \frac{4}{15\log 2}$.
\end{proof}

\begin{proof}[Proof of Theorem \ref{lower_bound}] Recall that $p$ is the probability that the time $1$ is assigned to each edge. Let $\rho=\min \left\{p,1-p\right\}$. The configuration provided in the proof of Theorem  \ref{lower_bound_construction} has its probability greater than or equal to $\rho^{t^2}$. 
Therefore 
\begin{eqnarray*}P\left(A_t\geq \alpha t\log t\right)\geq \rho^{t^2}=e^{t^2\ln \rho}.\end{eqnarray*}
Therefore we may take $\mu=-\ln \rho$. 

\end{proof}

\section{Performance analysis}
\noindent The algorithm was implemented in C++ and OpenCL. The hardware used has a quad core Intel i5 processor with clock speed of 3.5GHz and AMD Radeon R9 M290X graphic card with 2 gigabytes of memory. The graphic card has 2816 processing elements.

The table  provides a comparison of the performance of the algorithm on $4$ samples of three dimensional cubes with edges of lengths $50$, $75$, $100$, and $125$. The initial configuration for  each of the graphs assumes that there is water on the boundary of the cube, while the set $B$ is defined to be the center of the cube. The same program was executed on graphic card and on CPU. 
\vspace{0.3cm}
\begin{center}
\begin{tabular}{|l|l|l|}\hline
Graph & GPU time (s)& CPU time (s)\\ \hline
$50\times 50\times 50$ & 3& 10\\ \hline
$75\times 75\times 75$ & 8& 61\\ \hline
$100\times 100\times 100$ & 21& 275\\ \hline
$125\times 125\times 125$ & 117& 1540\\ \hline
\end{tabular}
\end{center} 
\vspace{0.3cm}

The graph that corresponds to the cube $100\times 100\times 100$ has $1000000$ vertices and $2970000$ edges, while the graph corresponding to the cube $125\times 125\times 125$ has $1953125$ vertices and $5812500$ edges. 

\section{Conclusion} 
The algorithm described in this chapter solves the constrained shortest path problem using parallel computing. It is suitable to implement on graphic cards and CPUs that have large number of processing elements. 
The algorithm is implemented in C++ and OpenCL and the parallelization improves the speed tenfold. 

The main idea is to follow the percolation of water through the graph and assign different qualities to drops that travel over different edges. Each step of the algorithm corresponds to a unit of time. It suffices to analyze only those vertices and edges through which the water flows. We call them active vertices and active edges. Therefore, the performance of the algorithm is tied to the sizes of these active sets. 

Theorem \ref{lower_bound} proves that it is possible to have at time $t$ an active set of size $O(t\log t)$. The proof of the theorem relied on constructing one such set. It is an open problem to find the average size of the active set at time $t$.

\begin{problem} If the weights and travel times of the edges are chosen independently at random, what is the average size of the active set at time $t$?
\end{problem}

 At some stages of the execution, the program needs additional memory to store phantom edges in the graph. It would be interesting to know how many phantom edges are allocated during a typical execution. 
This can be formally phrased  as an open problem.
\begin{problem}
If the weights and travel times of the edges are chosen independently at random, what is the average number of phantoms that need to be created during the execution of the algorithm?
\end{problem}

\section*{Acknowledgements} The author was supported by PSC-CUNY grants  $\#68387­-00 46$, $\#69723­-00 47$ and  Eugene M. Lang Foundation.

%%%%%%%%%%%%%%%%%%%%%%%% referenc.tex %%%%%%%%%%%%%%%%%%%%%%%%%%%%%%
% sample references
% %
% Use this file as a template for your own input.
%
%%%%%%%%%%%%%%%%%%%%%%%% Springer-Verlag %%%%%%%%%%%%%%%%%%%%%%%%%%
%
% BibTeX users please use
% \bibliographystyle{}
% \bibliography{}
%
\biblstarthook{ 
}

\end{document}